\documentclass{amsart}
\usepackage{ amsfonts,amssymb,amsmath,amsthm}
\usepackage{url}
\usepackage{enumerate}
\usepackage{graphicx}
\newtheorem{thrm}{Theorem}[section]
\newtheorem{lem}[thrm]{Lemma}

\theoremstyle{definition}

\newtheorem{remark}[thrm]{Remark}
\numberwithin{equation}{section}

\author[M. Ahanjideh]{Milad~Ahanjideh}
\address{
Department of Mathematics,
Faculty of Mathematical Sciences, Tarbiat Modares University, Tehran, Iran.}
\email{ahanjidm@gmail.com}
\author[E. Aboomahigir]{Elham~Aboomahigir}
\address{
Department of Mathematics,
Faculty of Mathematical Sciences, Tarbiat Modares University, Tehran, Iran.}
\email{mahigir.elham@gmail.com}

\begin{document}

\title[Hoffmann-Ostenhof's conjecture for claw-free cubic graphs]{Hoffmann-Ostenhof's conjecture for claw-free cubic graphs}

\begin{abstract}
Hoffmann-Ostenhof's Conjecture states that the edge set of every connected cubic graph can be decomposed into a spanning tree, a matching and a $2$-regular subgraph. In this paper, we show that the conjecture holds for claw-free cubic graphs.
\end{abstract}
\maketitle

\section{Introduction} \label{sect1}
Graphs are finite without loops and multiple edges throughout this paper.
 Let $G$ be a finite undirected graph with the vertex set $V(G)$ and the edge set $E(G)$. For a vertex $v\in V(G)$, the degree of $v$ in $G$ is denoted by $d_G(v)$.
 An edge joining vertices $u$ and $v$ is denoted by $uv$. Here, $N_G(v)$ denotes the set of all neighbours of $v$. 
 The complete graph of order $n$ is denoted by $K_n$. The complete graph on four vertices minus one edge is called a \textit{diamond}. A graph is \textit{cubic} in which all vertices have degree three and a \textit{subcubic} graph is a graph in which each vertex has degree at most three. A graph is \textit{claw-free} if it has no induced subgraph isomorphic to $K_{1,3}$. A cycle $C$ is called \textit{chordless} if $C$ has no cycle chord. For every $S\subseteq V(G)$, the graph obtained by removing all vertex in $S$ and all associated incident edges is denoted by $G\setminus S$. 
 If $X$ is a subset of the edges, then $G \setminus X$ is the graph obtained by removing all edges in $X$ from $G$. A \textit{cut-edge} of a connected graph $G$ is an edge $e\in E(G)$ such that $G\setminus e$ is disconnected.
 \\
 Now, we adopt some terminology from \cite{Henning} in which we need for the proof of the main theorem.
 For convenience, we repeat their definitions here.\\ Let $k\geq 2$ be an integer, a {\it diamond-necklace} $N_k$ with $k$ diamonds is a connected cubic graph constructed as follows. Take $k$ disjoint copies $D_1, D_2,\ldots , D_k$ of a diamond, where $V(D_i) = \{a_i, b_i, c_i, d_i\}$ and $a_id_i$ is the missing edge in $D_i$. Let $N_k$ be obtained from the disjoint union of these $k$ diamonds by adding the edges $\{d_{i}a_{i+1} |\  i = 1, 2,\ldots , k-1\}$ and adding the edge $d_ka_1$, see Figure \ref{from the left side: A diamond-necklace $N_7$; A diamond-bracelet $B_6$; A diamond-chain $L_2$}.\\ For $k \geq 1$, a {\it diamond-bracelet} $B_k$ with $k$ diamonds is defined as follows. Let $B_k$ be obtained from a diamond-necklace $N_{k+1}$ with $k+1$ diamonds $D_1, D_2,\ldots , D_{k+1}$ by removing the diamond $D_{k+1}$ and adding a triangle $T$ with $V(T) = \{a, b, c\}$, and adding two edges $ba_1$ and $ad_k$, see Figure \ref{from the left side: A diamond-necklace $N_7$; A diamond-bracelet $B_6$; A diamond-chain $L_2$}.\\ For $k\geq 1$, a {\it diamond-chain} $L_k$ with $k$ diamonds is defined as follows. 
Let $L_k$ be obtained from $k$ disjoint copies $D_1, D_2,\ldots , D_k$ of a diamond by adding the edges $\{d_ia_{i+1}|\ i=1,\ldots, k-1\}$. After that we add two disjoint triangles $T_1$ and $T_2$ and add one edge joining $a_1$ to a vertex of $T_1$, and one edge joining $d_k$ to a vertex of $T_2$, see Figure \ref{from the left side: A diamond-necklace $N_7$; A diamond-bracelet $B_6$; A diamond-chain $L_2$}.

\begin{figure}[h!]
\centering
\includegraphics[scale=0.5]{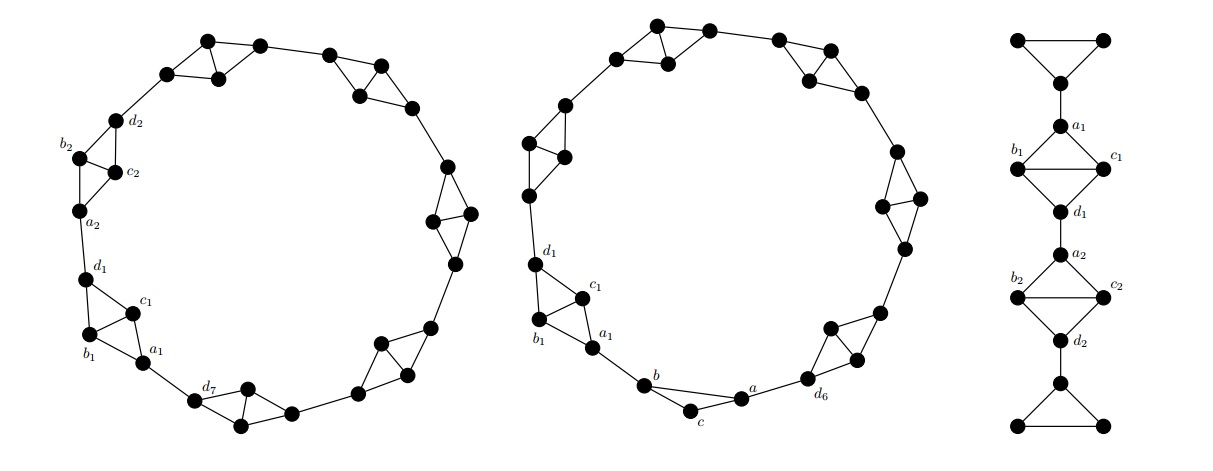}
\caption{\footnotesize {From the left side: A diamond-necklace $N_7$; A diamond-bracelet $B_6$; A diamond-chain $L_2$.}}
 \label{from the left side: A diamond-necklace $N_7$; A diamond-bracelet $B_6$; A diamond-chain $L_2$}
\end{figure}
The following conjecture was posed in \cite{Hoff1} and appeared as a problem in BCC22 \cite{Cameron}:\\
\\
\textbf{Hoffmann-Ostenhof's Conjecture.} Let $G$ be a connected cubic graph. Then the edges of $G$ can be decomposed into a spanning tree, a matching and a $2$-regular subgraph.\\
\\
Note that the spanning tree or the $2$-regular subgraph cannot be empty, however the matching may be empty.
An edge decomposition of a graph $G$ is called a {\it good decomposition}, if the edges of $G$ can be decomposed into a spanning tree, a matching and a $2$-regular subgraph. A graph is called {\it good} if it has a good decomposition. Throughout, we use $\{T, M, O\}$ to denote the spanning tree, the matching and the $2$-regular subgraph of the good decomposition of $G$, respectively. 
\\ Hoffmann-Ostenhof's Conjecture is known to be true for some families of cubic graphs. Kostochka \cite{Kostochka} showed that the Petersen graph, the prism over cycles, and many other graphs are good. Bachstein \cite{Bachstein} proved that every $3$-connected cubic graph embedded in torus or Klein-bottle is good. Furthermore, Ozeki and Ye \cite{Ozeki} proved that $3$-connected cubic plane graphs are good. Akbari et. al. \cite{Akbari} showed that hamiltonian cubic graphs are good. Also, it has been proved that traceable cubic graphs are good \cite{Abdolhosseini}. In $2017$, Hoffmann-Ostenhof et. al. \cite{Hoff3} proved that planar cubic graphs are good.
Recently in  \cite{aaa} the authors proved that claw-free subcubic graphs and $4$-chordal subcubic graphs have the same property.
In this paper, we are interested in finding a good decomposition of claw-free cubic graphs. We prove the following theorem.
\\
\\
\textbf{Main Theorem.}\label{B}
Let $G$ be a connected claw-free cubic graph. Then $G$ is good.

\section{Preliminary Results}
In order to prove the main theorem, we need to know a few basic properties about  claw-free cubic graphs.
\begin{lem}\rm \cite [Claim A]{Henning}
\label{lem1}
If $G\neq K_4$ is a connected claw-free cubic graph, then the vertex set $V(G)$ can be uniquely partitioned into sets each of which induces a triangle or a diamond in $G$.
\end{lem}
In the next two lemmas, we discuss some properties of a good decomposition of a claw-free cubic graph.
\begin{lem}
\label{lem2}
Let $G$ be a connected claw-free cubic graph which has a good decomposition and let $C$ be a cycle in $G\setminus T$, where $T$ is a spanning tree of the good decomposition of $G$. Then $C$ is either a triangle or the following statements hold for $C$:
\begin{itemize}
\item[(i)] $C$ is a chordless cycle. 
\item[(ii)] There is no vertex of $C$ in a diamond of $G$.
\item[(iii)] $C$ is an even cycle.
\end{itemize}
\end{lem}
\begin{proof}
Assume that $C$ is not a triangle.
\begin{itemize}
\item[(i)] Suppose that $C$ contains two vertices $a$ and $b$ such that $ab$ is a chord in $C$. Since $d_G(a)=d_G(b)=3$, the spanning tree $T$ does not contain $a$ and $b$, a contradiction. 
\item[(ii)] 
Suppose that (ii) is false. Since $G$ is cubic, the cycle $C$ should pass through of $3$ or $4$ vertices of a diamond. It is easy to see that the spanning tree $T$ does not contain some vertices of this diamond, a contradiction. 
\item[(iii)]
 Clearly, every vertex of a claw-free cubic graph lies in a triangle. Assume that $a_1$, $a_2$ and $a_3$ are three consecutive vertices in $C$ such that $a_1a_2a_3$ is a path. By (i), $a_1$ and $a_3$ are not adjacent. Thus it is easy to see that two consecutive vertices of every three consecutive vertices of $C$ lie in a triangle. Note that there is no vertex in $C$ lying in a diamond. So there are $|V(C)|/2$ edges belonging to triangles of $G$ and hence $|V(C)|$ is even.
 \end{itemize} 
\end{proof}
\begin{lem}
\label{lem2.3}
Let $G$ be a connected claw-free cubic graph with the good decomposition $\{T, M, O\}$. If $C\neq K_3$ is a cycle in the $2$-regular subgraph $O$, then for the fix edge $xy$ of $C$, we can find another good decomposition such that $xy$ belongs to either $T$ or $M$.
\end{lem}
\begin{proof}
Assume that $C$  is a cycle in $O$ containing $xy$. By Lemma \ref{lem2}, $C$ is an even cycle. For the convenience, let us recall $x$ by $b_1$ and $y$ by $a_2$ and call every triangle $a_it_ib_i$ with one edge on $C$ by $C_i$. If $b_1a_2$ does not lie in a triangle, then there are two distinct triangles $C_1$ and $C_2$ such that $b_1$ belongs to $C_1$ and $a_2$ belongs to $C_2$. Note that $t_i$ is not adjacent to $t_j$  for $i, j$, where $1\leq i, j\leq |V(C)|/2$, otherwise we have a spanning forest instead of a spanning tree in $G$. Now, assume that the number of triangles with one edge on $C$ is $2k$. We change the decomposition of $G$ such that $b_1a_2$ lies in the spanning tree as follows. We consider $T'=T\cup \{b_1a_2, b_{2k}a_1, a_{2i}t_{2i}, t_{2i}b_{2i}, b_ja_{j+1}|\ i=1,\ldots, k,\  j=1,\ldots,2k-1\}$, $M'=M\cup\{a_{2i}b_{2i}| \ i=1,\ldots, k\}$ and $O'= (O\setminus\{b_1a_2\})\cup \{C_{2i-1}| \ i=1,\ldots, k\}$ as a new good decomposition into a spanning tree, a matching and a $2$-regular subgraph of $G$.  
If the number of triangles with one edge on $C$ is $2k-1$, then similar to the above argument, we change the decomposition of $G$ such that $b_1a_2$ lies in a matching.
We consider $T' =T \cup \{ a_1b_1, t_1b_1, a_2t_2, a_2b_2, b_{2k-1}a_1, a_{2i}t_{2i}, t_{2i}b_{2i}, b_ja_{j+1} |\ i=2,\ldots , k-1,\  j=2,\ldots ,2k-2\}$,  $M'=M \cup \{a_1t_1, b_1a_2, t_2b_2, a_{2i}b_{2i}|\ i=2,\ldots , k-1\}$ and  $O' = (O\setminus\{b_1a_2\}) \cup \{C_{2i-1}| \ i=2,\ldots , k \}$  as a new good decomposition of $G$. See Figure \ref{ from the left side: Case 6, Cycle C with 6 triangle on it ;  Case 6, Cycle C with 5 triangle on it}(b).\\
Now, Assume that $b_1a_2$ lies in a triangle. Then consider one of incident edges with $xy$ on cycle $C$ and fix it. Then we apply the above method. We can see at once that $xy$ belongs to $T$ or $M$.
\begin{figure}[h!]
\centering
\includegraphics[scale=0.5]{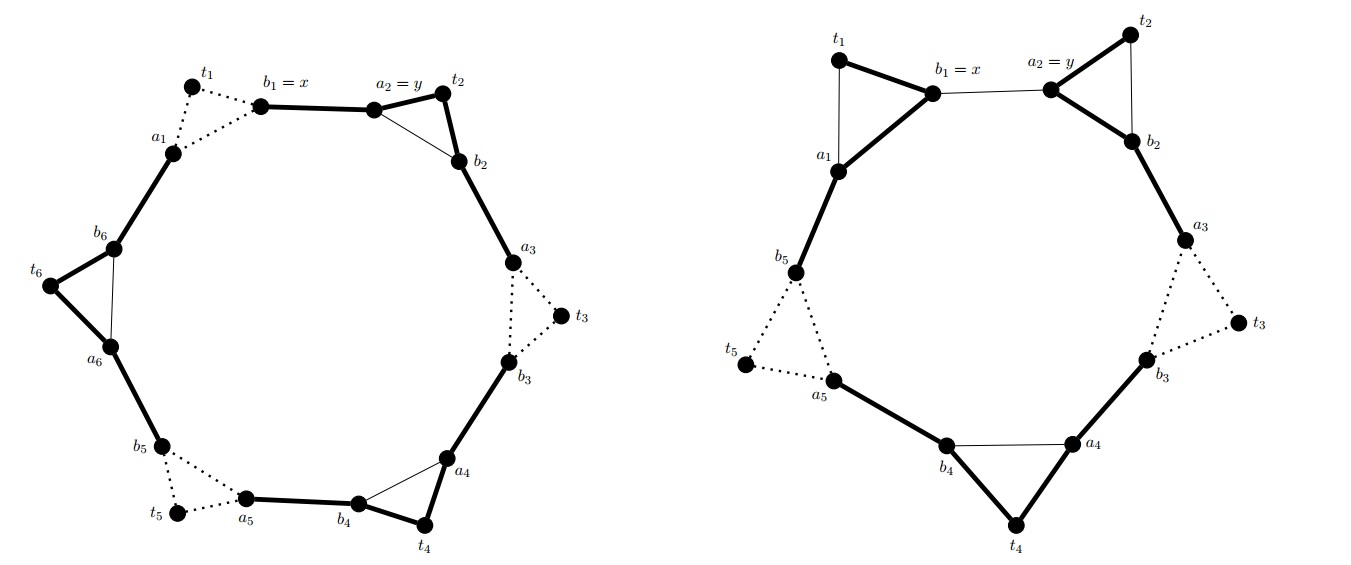}
\caption{\footnotesize {From the left side: a, b.}}
 \label{ from the left side: Case 6, Cycle C with 6 triangle on it ;  Case 6, Cycle C with 5 triangle on it}
\end{figure}
\end{proof}
\section{Proof of the Main Theorem}
\begin{proof}
We proceed by induction on the order $|V(G)|=n$ of $G$.
Obviously, both $K_4$ and the Cartesian product $C_3\square K_2$ have the desired decompositions.
Now, we divide the proof into two steps:\\
\textbf{Step 1.} Let $G$ be a claw-free cubic graph of order $n$ with no diamond.\\
Let $n\geq 8$ and the hypothesis holds for all claw-free cubic graphs of order less than $n$. Assume that $v_1 v_2 v_3$ is a triangle in $G$. Let $u_1$ and $u_2$ be the other neighbours of $v_1$ and $v_2$, respectively. If $u_1$ and $u_2$ are adjacent, then since $G$ is  claw-free, $u_1$ and $u_2$ have a common neighbour $u_3$. If $u_3$ and $v_3$ are adjacent, then $G$ is the Cartesian product $C_3\square K_2$ and it has a good decomposition. So,  assume that $x\neq u_3$ is a neighbour of $v_3$ and $y\neq v_3$ is a neighbour of $u_3$. It is easy to see that $x\neq y$, otherwise there is a claw in $x$. If $x$ and $y$ are not adjacent, then we put $A =\{u_1, u_2, u_3, v_1, v_2, v_3\}$ and consider $G\setminus A$. Join $x$ to $y$ and call the resulting claw-free cubic graph $G^A$. Since $|V(G^A)|<|V(G)|$, by induction hypothesis $G^A$ has a good decomposition $\{T^A , M^A, O^A\}$, where $T^A$ is a spanning tree of $G^A$, $M^A$ is a matching and $O^A$ is a $2$-regular subgraph of $G^A$. We extend this decomposition of $G^A$ to $G$. It is clear that $G^A$ is connected, so we have the following three cases:
\begin{itemize}
\item[i.] If $xy$ lies in the spanning tree $T^A$, then since the spanning tree $T^A$ saturates vertices $x$ and $y$, we can consider $T=(T^A\setminus\{xy\}) \cup \{xv_3, v_3v_1, v_1v_2, v_2u_2, u_2u_1\\,u_1u_3, u_3y \}$ and $M= M^A\cup\{v_3v_2, v_1u_1, u_2u_3\}$ as a spanning tree and a matching of $G$, respectively and we keep the cycles in the good decomposition of $G^{A}$, see Figure \ref{22} (a).
\item[ii.] If $xy$ lies in the matching $M^A$, then we define $T=T^A \cup\{xv_3, v_1v_3, v_2v_3, u_1u_3,\\u_2u_3,u_3y \}$, $M= M^A\setminus \{xy\}$ and $O= O^A \cup\{v_1v_2u_2u_1\}$ as a good decomposition of $G$, see Figure \ref{22} (b).
\end{itemize}
Before we discuss the third case, we need the following claim.\\
\\
\textbf{Claim 1.} Let $C$ be a cycle in $G^A$ containing $xy$, then $|V(C)|>3$.  
\label{Claim 1}
\begin{proof}
Assume that $|V(C)|=3$, then there is $t_1\in N_{G^A}(x)\cap N_{G^A}(y)$. Let $a_1\in N_{G^A}(x)\setminus \{t_1\}$ and $a_2\in N_{G^A}(y)\setminus \{t_1, a_1\}$. Note that if $a_1$ and $a_2$ are same, then there are claws in $G$. Since $G$ is cubic, $t_1$ can be adjacent to at most one of the two vertices $a_1$ and $a_2$.
Therefore $G$ has at least one claw containing $x$ or $y$, a contradiction. So $|V(C)|>3$ and in particular $xy$ is not in a triangle in $G^A$.
\end{proof}
\begin{itemize}
\item[iii.] If $xy$ lies in the cycle, then by Claim 1, $|V(C)|>3$ and hence by Lemma \ref{lem2.3}, this case will be converted to either Case $i$ or Case $ii$.
\end{itemize}
\begin{figure}[h!]
\centering
\includegraphics[scale=0.5]{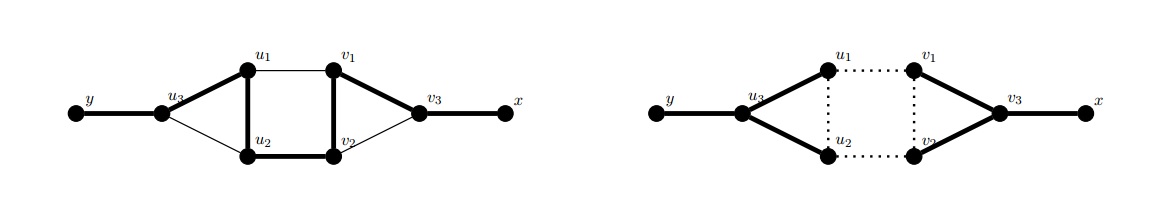}
 \caption{\footnotesize {From the left side: a; b.}}
 \label{22}
\end{figure}
Now, if $x$ and $y$ are adjacent, then for avoiding of a claw, there is $b\in N_G(x)\cap N_G(y)$, see Figure \ref{ddd}. Note that triangle $v_1v_2v_3$ join to another triangle $xyb$ by the edge $v_3x$. We investigate this case in the below by replacing $x$ with $u_3$, $y$ with $u_1$ and $b$ with $u_2$.\\
\begin{figure}[h!]
\centering
\includegraphics[scale=0.5]{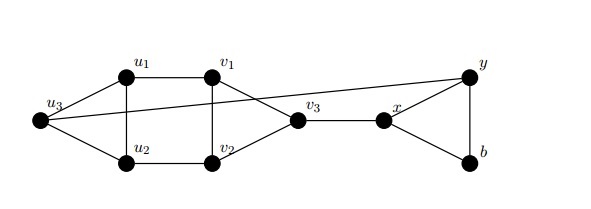}
 \caption{\footnotesize {If $x$ and $y$ are adjacent.}}
\label{ddd}
\end{figure}
Now, assume that triangle $v_1v_2v_3$ join to another triangle $u_1u_2u_3$ by the edge $u_3v_3$. Let $x, y, w$ and $z$ be the neighbours of $v_1, v_2, u_1$ and $u_2$, respectively. If either $x=u_1$ or $x=u_2$, then we have done before. 
 Without loss of generality, we can assume that $x, y, w$ and $z$ are different from $u_1, u_2, v_1$ and $v_2$.  Since there is no diamond in $G$, $x\neq y$ and $w\neq z$. If $x$ and $y$ or $w$ and $z$ are adjacent, then we have done before. So, assume that $x$ and $y$ ($w$ and $z$) are not adjacent. If there is $t\in N_G(x)\cap N_G(y)$, then there is a claw in vertices $x$ or $y$. If $x=w$ or $x=z$, then we have a claw in $x$. So, assume that $x$, $y$, $w$ and $z$ are not adjacent. Put $B= \{u_1, u_2, u_3, v_1, v_2, v_3\}$ and consider $G\setminus B$. Join $x$ to $y$ and $w$ to $z$, we denote the new graph $G^{B}$. It is obvious that $G^B$ is a claw-free cubic graph and $|V(G^B)|<|V(G)|$.

If $G$ is $2$-edge connected, then $G^B$ is connected. So by induction hypothesis $G^B$ has a good decomposition $\{T^B, M^B, O^B\}$, where $T^B$ is a spanning tree of $G^B$, $M^B$ is a matching and $O^B$ is a $2$-regular subgraph. We extend this decomposition of $G^B$ to $G$. The proof falls into six cases:
\begin{itemize}
\item[Case 1.] If $xy$ and $wz$ belong to $M^B$, then since the spanning tree $T^B$ saturates vertices $x, y, w$ and $z$, consider $T=T^B \cup \{xv_1, v_1v_2, v_2v_3, v_3u_3, wu_1, u_2z \}$, $M=(M^B\setminus\{xy, wz\})\cup \{v_1v_3, v_2y\}$ and $O= O^B\cup \{u_1u_2u_3\}$ as a good decomposition of $G$, see Figure \ref{ from the left side: a, b, c, d} (a).
\item[Case 2.] If $xy$ and $wz$ belong to $T^B$, then we can consider
$T=(T^B\setminus \{xy, wz\}) \cup \{xv_1, v_1v_3, v_3v_2, v_2y, wu_1, u_1u_3, u_3u_2, u_2z \}$, $M= M^{B}\cup\{v_1v_2, u_1u_2, u_3v_3\}$ and $O= O^B$ as a good decomposition of $G$, see Figure \ref{ from the left side: a, b, c, d} (b).
\item[Case 3.] If $xy\in T^B$ and $wz\in M^B$, then we consider $T=(T^B\setminus \{xy\}) \cup \{xv_1, v_1v_2, v_2v_3,\\v_2y, v_3u_3, wu_1, u_2z \}$, $M=(M^B\setminus\{wz\})\cup \{v_1v_3\}$ and $O= O^B\cup \{ u_1u_2u_3\}$ as a good decomposition of $G$.
\item[Case 4.] If $xy\in T^B$ and $wz\in O^B$, then we can consider $T=(T^B\setminus \{xy\}) \cup \{xv_1, v_1v_3, v_3v_2, v_2y, v_3u_3, u_3u_1, u_3u_2\}$, $M=M^B\cup\{v_1v_2\}$ and  $O= (O^B\setminus \{wz\})\cup \{wu_1,u_1u_2,u_2z\}$ as a good decomposition of $G$, see Figure \ref{ from the left side: a, b, c, d} (c).
\item[Case 5.] If $xy\in M^B$ and $wz\in O^B$, then we have $T=T^B \cup \{xv_1,v_1v_2, v_2v_3, v_3u_3, u_3u_1,\\ u_3u_2 \}$, $M=(M^B\setminus \{xy\})\cup \{v_1v_3, v_2y\}$ and $O= (O^B\setminus \{wz\})\cup \{wu_1,u_1u_2,u_2z\}$ as a good decomposition of $G$, see Figure \ref{ from the left side: a, b, c, d} (d).
\begin{figure}[h!]
\centering
\includegraphics[scale=0.5]{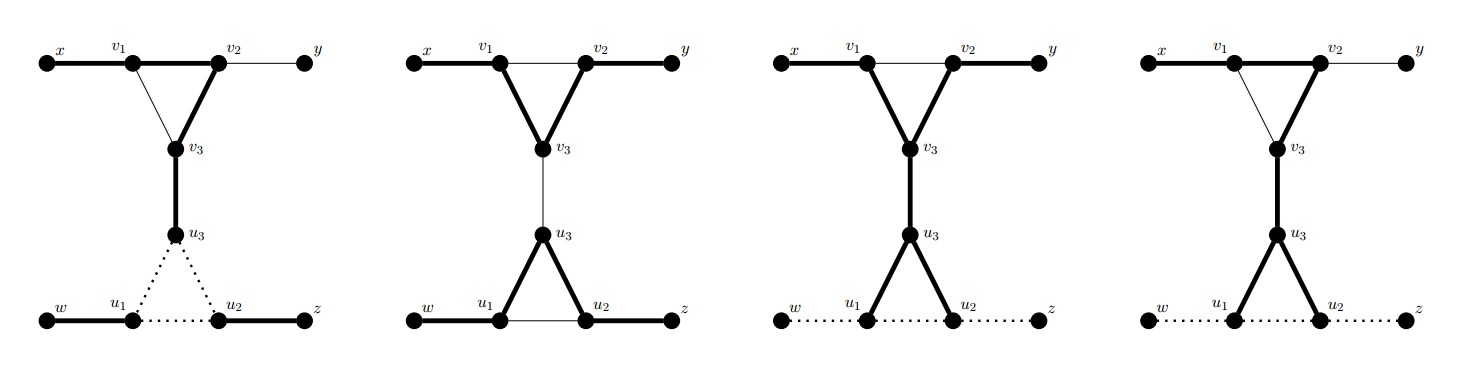}
 \caption{ \footnotesize {From the left side: a, b, c, d.}}
 \label{ from the left side: a, b, c, d}
\end{figure}
\item[Case 6.] If $xy$ and $wz$ belong to $O^B$, then similar to that in the proof of Claim 1 we have $C\neq K_3$. So by Lemma \ref{lem2.3}, this case will be converted to either Case $4$ or Case $5$.
\end{itemize}
Now, assume that $G^B$ is not connected. Let ${G_1^B}$ and $G_2^B$ be the connected components of $G^B$. So by induction hypothesis ${G_1^B}$ and ${G_2^B}$ have good decompositions $\{T_1^B, M_1^B, O_1^B\}$ and $\{T_2^B, M_2^B, O_2^B\}$. From this we want to find a new decomposition $\{T^{B}, M^B, O^B\}$ of $G^B$, where $T^{B}$ is a spanning forest, $M$ is a matching and $O^B$ is a $2$-regular subgraph. Define $T^{B} = T_1^B\cup T_2^B$, $M^B = M_1^B\cup M_2^B$ and $O^B = O_1^B\cup O_2^B$. For extending this decomposition of $G^B$ to $G$, we make slight changes in some cases for when $G^B$ is connected. Now, we investigate those cases here. We remove $xy$ and $wz$ from the decomposition of $G^B$. In Case $1$, a good decomposition of $G$ is obtained by adding the set of edges $\{xv_1, v_1v_2, v_2v_3, v_3u_3, u_3u_2, u_2u_1, u_1w\}$ to $T^{B}$ and $\{v_1v_3, u_1u_3, v_2y, u_2z\}$ to $M^{B}\setminus \{xy, wz\}$. In Case $2$, we add the set of edges $\{xv_1, v_1v_3, v_2v_3, v_2y, v_3u_3, wu_1, u_1u_3, u_3u_2, u_2z \}$ to $T^{B}\setminus \{xy, wz\}$ and $\{v_1v_2, u_1u_2\}$ to $M^{B}$ to obtain a spanning tree and a matching in a good decomposition of $G$. We keep the cycles in the good decomposition of $G^B$. In Case $3$, by adding the set of edges $\{xv_1, v_1v_2, v_2y, v_2v_3, v_3u_3, u_3u_2, u_2u_1, u_1w\}$ to $T^{B}\setminus \{xy\}$ and $\{v_1v_3, u_1u_3, u_2z\}$ to $M^{B}\setminus \{wz\}$, we obtain a good decomposition of $G$.
Other cases are similar to Cases 4-6 that the graph $G$ is $2$-edge connected.\\
Note that by applying the proof of Claim $1$, we can deduce that $G^{B}$ has no diamond.\\
\textbf{Step 2.} Let $G$ be a claw-free cubic graph of order $n\geq8$ with at least one string which contains $k\geq 1$ diamonds.
First, assume that $G$ is a diamond-necklace with $k\geq 2$ diamonds. Because $G$ is hamiltonian, according to the proof of  Theorem $9$ in \cite{Akbari}, $G$ has the desired decomposition.
Now, assume that $G$ is a claw-free cubic graph of order $n$ containing a string of at least two consecutive diamonds. Without loss of generality, consider a string of the maximum number of diamonds with starting and ending vertices $w$ and $z$. It is easy to check that $w\neq z$. By Lemma \ref{lem1}, $w$ and $z$ belong to triangles or diamonds. If at least one of $w$ or $z$ belongs to a diamond, then it leads to a contradiction with our assumption on the maximum number of diamonds in string. Now, we consider the following cases:
\begin{itemize}
\item [a.] Let $w$ and $z$ belong to a common triangle. That is, we have a diamond-bracelet $B_{k}$ in $G$.
\item [b.]  Let $w$ and $z$ belong to two distinct triangles.
\end{itemize}
 
Since the proof of both cases are similar, we prove only part (a) and the proof of part (b) is left to the reader.
we consider $G\setminus \{a_i, b_i, c_i, d_i\ |\ i=2,\ldots, k\}$ and then join $d_1$ to $z$. Let us call this new graph $G'$. By induction hypothesis, $G'$ has a good decomposition $\{T', M', O'\}$. We extend this good decomposition to $G$.
 If $d_1z$ lies in the spanning tree $T'$, then we define $T=T'\cup  \{a_ib_i, b_ic_i, c_id_i| \ 2\leq i\leq k\}\cup\{d_ja_{j+1}| \ 1\leq j\leq k-1\} \cup \{ d_kz\}$, $M=M'\cup \{a_ic_i, b_id_i |\ 2\leq i\leq k\}$ and we keep the cycles in the good decomposition of $G'$. If $d_1z$ lies in the matching $M'$, then we define $T=T'\cup \{ d_kz, b_2d_2, c_2d_2\}\cup  \{a_ib_i, b_ic_i, c_id_i| \ 3\leq i\leq k\}\cup\{d_ja_{j+1}| \ 1\leq j\leq k-1\}$,  $M=M'\cup \{a_ic_i, b_id_i |\ 3\leq i\leq k\}$ and $O= O'\cup \{a_2b_2c_2\}$, see Figure \ref{$d_1z$ lies in the spanning tree $T′$ or on the matching $M'$.}. Now, let us consider the case in which $d_1z$ lies in a cycle. Since $G$ is cubic, the edge $wa_1$ lies in every possible cycle containing $d_1z$ in the decomposition of $G'$. It is easy to check that of $a_1, b_1, c_1$ and $d_1$ are not covered by the spanning tree of $G'$, a contradiction.\\
\begin{figure}[h!]
\centering
\includegraphics[scale=0.5]{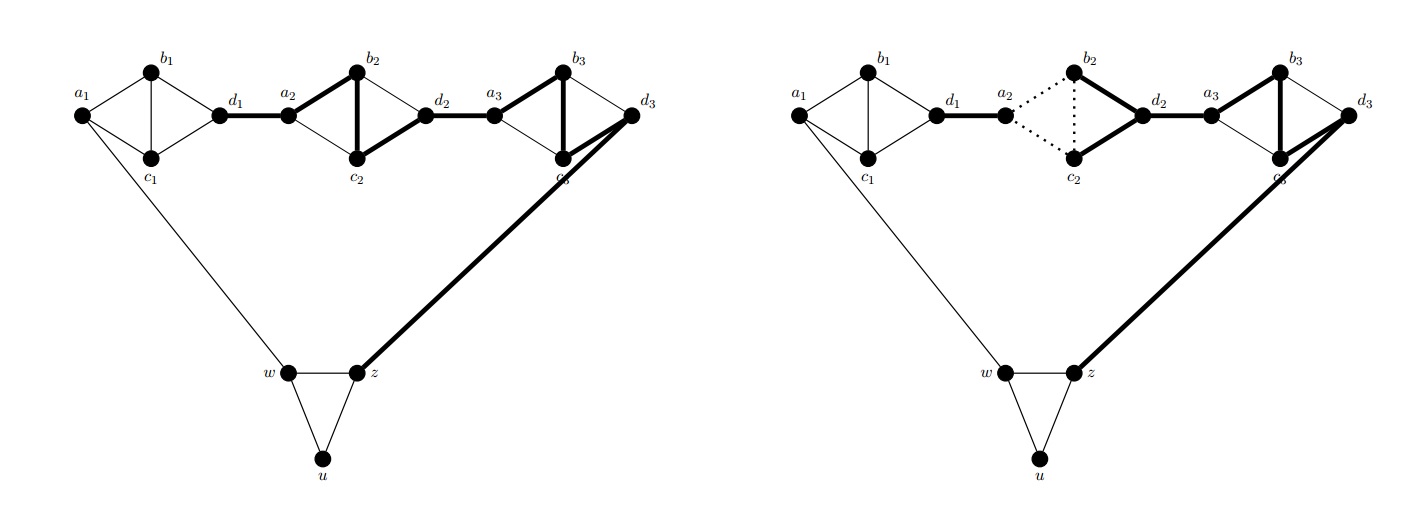}
 \caption{\footnotesize {From the left side: a) $d_1z$ lies in the spanning tree $T'$; b) $d_1z$ lies in the matching $M'$.}}
 \label{$d_1z$ lies in the spanning tree $T′$ or on the matching $M'$.}
\end{figure}
Here, we examine the case in which $G$ has only one diamond. First, assume that $w$ and $z$ lie in two distinct triangles. Remove the diamond and the edges $wa_1, d_1z$, then join $w$ to $z$ and call this new cubic graph $G'$. The graph $G'$ is a claw-free cubic graph and 
$|V(G')| < |V (G)|$. So by Step 1, $G'$ has a good decomposition $\{T', M', O'\}$. If $wz$ lies in the spanning tree $T'$, then we define $T=T'\cup \{wa_1, a_1b_1, b_1c_1, c_1d_1, d_1z\} $, $M=M'\cup \{a_1c_1, b_1d_1\}$ and we keep the cycles in the good decomposition of $G'$. If $wz$ lies in the matching $M'$, then we define $T=T'\cup\{wa_1, a_1b_1, a_1c_1, d_1z\} $, $O= O'\cup \{b_1c_1d_1\}$ and we keep the matching in the good decomposition of $G'$, see Figure \ref{ from the left side: $wz$ lies in the spanning tree the matching ;  $wz$ lies in the matching}. Assume that $wz$ lies in a cycle $C$, by Step $1$, the graph $G'$ has a good decomposition $\{T', M', O'\}$. It is obvious that $|V(C)|>3$. So by Lemma \ref{lem2.3}, we can convert this case to previous cases, that is the cases that $wz$ lies in either a spanning tree or a matching.

\begin{figure}[h!]
\centering
\includegraphics[scale=0.5]{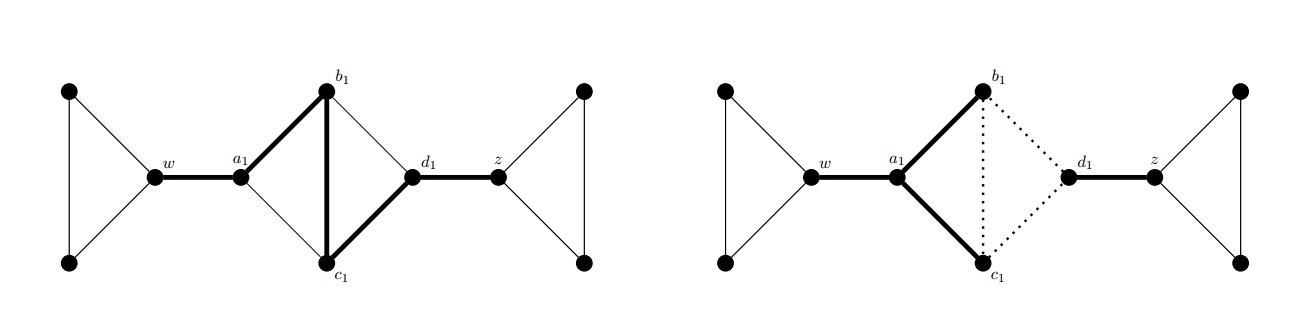}
 \caption{\footnotesize {From the left side: a) $wz$ lies in the spanning tree $T'$ ; b) $wz$ lies in the matching $M'$.}}
\label{ from the left side: $wz$ lies in the spanning tree the matching ;  $wz$ lies in the matching}
\end{figure}
\begin{figure}[h!]
\centering
\includegraphics[scale=0.5]{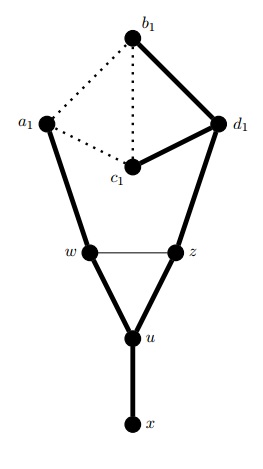}
 \caption{\footnotesize {Only one diamond, $w$ and $z$ lie in a common triangle.}}
 \label{Only one diamond, $w$ and $z$ lie in a common triangle}
\end{figure}
Now, assume that $w$ and $z$ are on a common triangle. Assume that $u$ is  a common neighbour of $w$ and $z$. Let another adjacent vertex to $u$ be $x$. It is obvious that $xu$ is a cut-edge of $G$. Assume that $G^x$ is a connected component of $G\setminus\{xu\}$ containing $x$. Since $G^x$ is a connected claw-free subcubic graph, by \cite[Theorem 2.1]{aaa}, the graph $G^x$ has a good decomposition $\{T^x, M^x, O^x\}$. Now, we extend this good decomposition to $G$.  Add the set of edges $\{xu, uw, uz, wa_1, zd_1, d_1c_1, d_1b_1\}$ to $T^x$, $wz$ to $M^x$ and the triangle $a_1b_1c_1$ to $O^x$ to obtain a good decomposition into the spanning tree,  the matching and the $2$-regular subgraph of $G$, see Figure \ref{Only one diamond, $w$ and $z$ lie in a common triangle}.
\end{proof}
\begin{remark}
{\rm Note that in the main theorem, the matching $M$ can be empty. We can see at once that the main theorem does not hold for subcubic graphs, a cycle is a counterexample.}
\end{remark}
\proof[Acknowledgements]
The authors would like to thank Professor Saieed Akbari for a discussion on this topic and for his comments.

\end{document}